\documentclass[11pt]{article}

\usepackage[T1]{fontenc}
\usepackage{lmodern}
\usepackage[margin=1in]{geometry}
\usepackage{amsmath,amssymb,amsthm}
\usepackage{booktabs}
\usepackage{longtable}
\usepackage{microtype}
\usepackage[hidelinks]{hyperref}
\hypersetup{
  pdftitle={Shuffle squares in smooth binary words and the Kolakoski sequence},
  pdfauthor={Kory Kilpatrick},
  pdfsubject={Shuffle squares, smooth words, and the Kolakoski sequence},
  pdfkeywords={shuffle square, smooth word, Kolakoski sequence, combinatorics on words}
}

\newtheorem{theorem}{Theorem}[section]
\newtheorem{lemma}[theorem]{Lemma}
\newtheorem{corollary}[theorem]{Corollary}
\newtheorem{proposition}[theorem]{Proposition}
\theoremstyle{definition}

\theoremstyle{remark}
\newtheorem{remark}[theorem]{Remark}

\newcommand{\eps}{\varepsilon}

\newcommand{\K}{\mathcal K}
\newcommand{\A}{\mathcal A}

\title{Shuffle squares in smooth binary words\\
and the Kolakoski sequence}
\author{Kory Kilpatrick}
\date{July 2026; revised 28 September 2026}

\begin{document}

\maketitle

\begin{abstract}
A binary word is a shuffle square if its positions can be partitioned
into two identical subsequences.  A finite word over $\{1,2\}$ is smooth
if repeated finite run-length differentiation remains binary until the
empty word is reached.  This paper proves that every smooth binary word in which
both letters have even multiplicity is a shuffle square once its length
is at least $18$.  There are exactly $34$ exceptions, all of length at
most $16$.  The infinite part of the proof is a parity-return atom
induction: a smooth even-Parikh word factors uniquely into smooth parity
atoms, and every such atom has length at most $8$.  A finite,
independently checkable certificate handles lengths $18$ through $40$.
As a consequence, a nonempty prefix of the classical Kolakoski sequence
is a shuffle square if and only if both symbols occur an even number of
times and its length is not $4$ or $8$.  There are infinitely many such
prefixes.  In fact, the entire infinite Kolakoski sequence can be
partitioned into two identical infinite subsequences with monochromatic
blocks of uniformly bounded length.
\end{abstract}

\section*{Research contributions and AI disclosure}

Kory Kilpatrick supervised this project. Apart from the external
contributions credited in the paper, GPT-5.6 Sol and GPT-6-Astra
(OpenAI generative-AI systems) performed the mathematical research and
proof development, wrote and executed the computational code, and
drafted and revised the manuscript. Kilpatrick's contribution was
supervision; he does not claim to have personally derived the proofs or
written the code or text. Verification included exact computational
checks and separate AI-agent reviews; these do not constitute independent
human peer review or proof-assistant formalization. As the submitting
author, Kilpatrick takes responsibility for the content.

\section{Introduction}

Let $w=w_1\cdots w_{2m}$ be a finite word.  We call $w$ a
\emph{shuffle square} if its positions can be partitioned into two
increasing supports
\[
 i_1<\cdots<i_m,\qquad j_1<\cdots<j_m
\]
such that $w_{i_t}=w_{j_t}$ for every $t$.  Equivalently, $w$ is an
interleaving of two copies of a common root.  Shuffle squares were
introduced by Henshall, Rampersad, and Shallit~\cite{HRS}.  Their
recognition is NP-complete even over a binary alphabet~\cite{Bulteau}.

A necessary condition is that every letter occur an even number of
times.  Following recent terminology~\cite{GPR}, we call such a word a
\emph{tangram}.  He, Huang, Nam, and Thaper~\cite{HHNT} studied the enumeration of
shuffle squares and conjectured that asymptotically almost every binary
tangram is a shuffle square. He and Post~\cite{HePost} proved this
conjecture.  Here an exact finite-obstruction theorem is established inside the highly structured
factorial language of smooth words.

Grytczuk, Pawlik, and Ruci\'{n}ski asked whether even one
prefix of the classical Kolakoski sequence is a shuffle
square~\cite[Question~1]{GPR}.  The length-$16$ prefix
\[
 1221121221221121
\]
is one.  The result gives the stronger exact statement that, apart from
two short obstructions of lengths $4$ and $8$, Parikh parity is the only
obstruction for a Kolakoski prefix.

The proof proceeds through a result about the full language of finite smooth
words.  Smooth words, also called finite $C^\infty$-words, were
introduced to study the factors of Kolakoski; see
\cite{BBC,CassaigneHenry,FedouFici,Weakley}.  The language is factorial and
is currently a more tractable envelope for the still-mysterious factor
language of Kolakoski.

\begin{theorem}\label{thm:main}
Every smooth binary tangram of length at least $18$ is a shuffle square.
\end{theorem}

\begin{corollary}\label{cor:34}
There are exactly $34$ smooth binary tangrams that are not shuffle
squares.  All have length at most $16$.
\end{corollary}

\begin{theorem}\label{thm:kolakoski}
Let $P_n$ be the nonempty prefix of length $n$ of the classical
Kolakoski sequence.  Then $P_n$ is a shuffle square if and only if both
of its letter multiplicities are even and $n\notin\{4,8\}$.  In
particular, infinitely many Kolakoski prefixes are shuffle squares.
\end{theorem}

In independent work, Zwierzy\'nski~\cite{Zwierzynski} obtained the
smooth-word parity theorem, the same $34$ exceptions, and the Kolakoski
prefix characterization by a different method. His manuscript also
extends the results to classes of finitely differentiable words.
The proof presented here uses parity-return atoms.

The computer-verified component of the proof is a finite base; the
infinite step is established by the written atom argument and induction.  A canonical
$291{,}550$-byte certificate records every case, and a standalone
checker regenerates the complete smooth language through length $40$,
replays every positive witness, and exhausts every negative support
partition.

\section{Preliminaries}

\subsection{Finite differentiation}

Write a nonempty word over $\{1,2\}$ as a product of maximal runs
\[
 w=a_1^{r_1}a_2^{r_2}\cdots a_t^{r_t},
 \qquad a_i\ne a_{i+1}.
\]
It is \emph{differentiable} if $r_i\in\{1,2\}$ for every $i$.  Its
finite derivative $D(w)$ is the factor of $r_1\cdots r_t$ obtained by
deleting $r_1$ when $r_1=1$ and deleting $r_t$ when $r_t=1$; if these
deletions cross, the result is empty.  Put $D(\eps)=\eps$.  A word is
\emph{smooth} if repeated differentiation reaches $\eps$ and every
intermediate word is over $\{1,2\}$.  This is the standard finite
derivative on the alphabet $\{1,2\}$.

\begin{lemma}\label{lem:factor-derivative}
If $u$ is a factor of a differentiable word $w$, then $D(u)$ is a factor
of $D(w)$, possibly the empty factor.
\end{lemma}

\begin{proof}
Every interior run of $u$ is a complete run of $w$.  A boundary run of
$u$ is either a singleton, in which case the finite derivative deletes
it, or has length two, in which case it is a complete run of $w$.
Thus the entries retained by $D(u)$ are consecutive retained entries of
$D(w)$.  Deleting a singleton boundary run of $w$ cannot remove an
entry retained by $D(u)$.
\end{proof}

\begin{lemma}[Factoriality]\label{lem:factorial}
Every factor of a smooth word is smooth.
\end{lemma}

\begin{proof}
Induct on derivative height.  A factor $u$ of a smooth word $w$ is
differentiable because $w$ has no run of length greater than two.  By
Lemma~\ref{lem:factor-derivative}, $D(u)$ is a factor of the smooth word
$D(w)$, and is therefore smooth by induction.  Hence $u$ is smooth.
\end{proof}

\subsection{Two elementary shuffle facts}

\begin{lemma}[Concatenation]\label{lem:concat}
If $u$ and $v$ are shuffle squares, then $uv$ is a shuffle square.
\end{lemma}

\begin{proof}
Concatenate corresponding supports, shifting the supports in $v$ by
$|u|$.  Both resulting subsequences are the concatenation of a root of
$u$ and a root of $v$.
\end{proof}

For the finite certificate we use the following convenient sufficient
condition.  Scan a word from left to right with a FIFO queue.  A control
letter $U$ appends the current symbol; a control letter $D$ requires the
current symbol to equal the queue front and removes that front.

\begin{lemma}[FIFO certificate]\label{lem:fifo}
An empty-to-empty FIFO control certifies a shuffle square.
\end{lemma}

\begin{proof}
The symbols at the $U$ positions form the queue-input stream and those
at the $D$ positions form the queue-output stream.  FIFO matching makes
these two subsequences identical.
\end{proof}

\section{Parity-return atoms}

For a word $u$, define
\[
 \pi(u)=(|u|_1,|u|_2)\pmod 2\in\mathbb F_2^2.
\]
If $w$ is a tangram, cut after each nonempty prefix $u$ for which
$\pi(u)=0$.  The resulting factors are its \emph{parity atoms}.

\begin{lemma}\label{lem:unique}
Every tangram has a unique factorization into nonempty tangrams having
no proper nonempty tangram prefix.
\end{lemma}

\begin{proof}
The cut positions are forced: they are exactly the returns of the prefix
parity walk to $0$.
\end{proof}

\begin{lemma}\label{lem:atoms}
The smooth parity atoms are exactly
\[
\begin{split}
\A=\{&
11,22,1212,1221,2112,2121,\\
&121121,212212,12112212,21221121\}.
\end{split}
\]
In particular, every smooth parity atom has length at most $8$.
\end{lemma}

\begin{proof}
Consider an atom beginning with $1$.  If its second letter is $1$, the
atom is $11$.  Otherwise it begins with $12$, after which its parity
state is $(1,1)$.  Before the terminal return to $(0,0)$, every
subsequent aligned two-letter block is $11$ or $22$; the terminal block is $12$ or
$21$.  Hence every atom beginning with $1$ has the form
\[
 11\quad\text{or}\quad 12(11\mid22)^*(12\mid21).
\]
Atoms beginning with $2$ are their complements.

A differentiable word contains neither $111$ nor $222$.  In the
displayed regular form, the internal double blocks must therefore
alternate, the first one must be $11$, and the last one is determined
by the terminal segment.  Besides $1212$ and $1221$, the possibilities
are
\[
 A_k=1211(2211)^k21,\qquad
 B_k=1211(2211)^k2212,\qquad k\ge0,
\]
and their complements.  Direct run-length calculation gives
\[
 D(A_k)=1\,2^{2k+1}1,\qquad
 D(B_k)=1\,2^{2k+2}1.
\]
For $k\ge1$, the derivative contains $222$, so the atom is not twice
differentiable and hence is not smooth.  At $k=0$, the two remaining
atoms are $121121$ and $12112212$, together with their complements.
This gives exactly $\A$.
\end{proof}

\begin{remark}
The same ten-word set appears as Rao's $W_2$ smooth bifix
code~\cite{Rao}.  Lemma~\ref{lem:atoms} gives the direct parity-return
interpretation needed here.
\end{remark}

\section{Proof of the smooth-word theorem}

\begin{proposition}[Finite base]\label{prop:base}
Every smooth tangram of length
\[
 18,20,22,\ldots,40
\]
is a shuffle square.
\end{proposition}

\begin{proof}[Computer-assisted verification]
The certificate described in Section~\ref{sec:verification} contains an
empty-to-empty FIFO control for every such word.  The standalone checker
independently regenerates the complete set of smooth words at these
lengths and replays each control.
\end{proof}

\begin{proof}[Proof of Theorem~\ref{thm:main}]
Use strong induction on the even length $n$.  Proposition~\ref{prop:base}
supplies the base $18,20,\ldots,40$.

Let $n\ge42$, and factor a smooth tangram $w$ into parity atoms.  Each
atom is a factor of $w$, hence is smooth by
Lemma~\ref{lem:factorial}, and therefore has length at most $8$ by
Lemma~\ref{lem:atoms}.

Let $c$ be the first atom boundary with $c\ge18$.  The preceding
boundary is at most $16$, while the next atom has length at most $8$.
All atom boundaries are even, so
\[
 18\le c\le24.
\]
Write $w=uv$ with $|u|=c$.  Both $u$ and $v$ are smooth tangrams, and
\[
 |u|\ge18,\qquad |v|=n-c\ge42-24=18.
\]
They are both strictly shorter than $w$.  The induction hypothesis
makes them shuffle squares, and Lemma~\ref{lem:concat} makes $w=uv$ a
shuffle square.
\end{proof}

\section{Exact finite verification}\label{sec:verification}

Smooth words were regenerated by complete primitive expansion.  For a
fixed nonempty derivative parent, every primitive is obtained by
choosing the first run symbol and every allowed optional singleton
boundary run.  The height-one roots are
\[
 1,\quad2,\quad12,\quad21.
\]
Since the derivative strictly contracts nonempty words, induction on
height proves completeness of this generation.

\begin{longtable}{rrrr}
\caption{Exact finite counts.  The last column counts non-shuffle
tangrams.}\label{tab:counts}\\
\toprule
$n$ & smooth words & smooth tangrams & negatives\\
\midrule
\endfirsthead
\toprule
$n$ & smooth words & smooth tangrams & negatives\\
\midrule
\endhead
2&4&2&0\\
4&10&6&2\\
6&18&6&0\\
8&34&18&6\\
10&50&22&4\\
12&78&42&10\\
14&110&54&4\\
16&142&74&8\\
18&186&66&0\\
20&250&150&0\\
22&314&146&0\\
24&378&222&0\\
26&446&206&0\\
28&534&274&0\\
30&654&282&0\\
32&782&414&0\\
34&910&466&0\\
36&1038&542&0\\
38&1166&538&0\\
40&1302&674&0\\
\bottomrule
\end{longtable}

The canonical certificate contains one FIFO control for each of the
$4{,}170$ positive smooth tangrams through length $40$.  For every
claimed negative, the checker directly exhausts all equal-size support
partitions.  The separately implemented standalone checker imports no discovery code; it
regenerates $16{,}146$ smooth words, covers all $4{,}204$ smooth
tangrams, replays all positive controls, and exhausts all $34$
negatives.

The ancillary files accompanying this preprint are
\begin{center}
\begin{tabular}{ll}
\toprule
File & Purpose\\
\midrule
\texttt{smooth-even-base-certificate.json}
 & canonical controls and exception set\\
\texttt{verify\_smooth\_even\_certificate.py}
 & standalone checker\\
\texttt{smooth-even-exceptions.md}
 & human-readable exception table\\
\bottomrule
\end{tabular}
\end{center}
The accompanying \texttt{README.md} gives a direct command for Python
3.10 or later; the checker uses only the standard library.
The canonical certificate has size $291{,}550$ bytes and SHA-256
\begin{center}
\small\ttfamily
d04eb188c2529a67ac5ee7161fe83940fe316b3ac54f7d42fbdd59be7936b42f.
\end{center}

For completeness, the $34$ exceptions are specified by ten orbits.
Let $G$ be the four-element group generated by reversal and binary
complementation.  The exception set is the union of the $G$-orbits in
Table~\ref{tab:exceptions}.

\begin{table}[ht]
\centering
\small
\caption{Orbit representatives for all non-shuffle smooth tangrams.}
\label{tab:exceptions}
\begin{tabular}{lrr}
\toprule
Representative & Orbit size & Length\\
\midrule
\texttt{1221} &2&4\\
\texttt{12122112} &4&8\\
\texttt{12212112} &2&8\\
\texttt{1211211221} &4&10\\
\texttt{112122122112} &4&12\\
\texttt{121122122112} &4&12\\
\texttt{122121121221} &2&12\\
\texttt{11221211211221} &4&14\\
\texttt{1212212211211221} &4&16\\
\texttt{1221121121221221} &4&16\\
\bottomrule
\end{tabular}
\end{table}

The orbit sizes sum to $34$.  All orbit members were checked directly,
not inferred only by symmetry.  Table~\ref{tab:counts},
Theorem~\ref{thm:main}, and the explicit orbit table prove
Corollary~\ref{cor:34}.

\section{Kolakoski prefixes}

Let
\[
 \K=122112122122112112212\cdots
\]
be the classical Kolakoski sequence, whose run-length word is $\K$
itself, and let $P_n=K_1\cdots K_n$.  Every finite prefix is smooth:
the run-length derivative of a factor is a factor of the infinite
derivative after boundary trimming, and every infinite derivative of
$\K$ is again $\K$.

Set
\[
 S(n)=\sum_{i=1}^nK_i.
\]
If $a=|P_n|_1$ and $b=|P_n|_2$, then
\[
 n=a+b,\qquad S(n)=a+2b.
\]
Therefore $P_n$ is a tangram if and only if $n$ and $S(n)$ are both
even.

\begin{lemma}\label{lem:infinite-parity}
Infinitely many Kolakoski prefixes are tangrams.
\end{lemma}

\begin{proof}
Group $\K$ into aligned pairs
\[
 Q_j=(K_{2j-1},K_{2j}).
\]
The parity of $S(2m)$ toggles at pair $Q_j$ exactly when $Q_j$ is mixed,
that is, $12$ or $21$.  If there are infinitely many mixed pairs, the
parity returns to zero after every second one, giving infinitely many
tangram prefixes.

Suppose instead that only finitely many aligned pairs are mixed.  Every
sufficiently late aligned pair is then $11$ or $22$.  Since every run of
$\K$ has length at most two, consecutive late homogeneous pairs must
alternate, and each such pair is one complete run.  Hence all
sufficiently late run lengths equal $2$.  The run-length word, which is
$\K$, would be eventually all $2$, contradicting the fact that every run
of $\K$ has length at most two.  Thus mixed aligned pairs are infinite.
\end{proof}

\begin{proof}[Proof of Theorem~\ref{thm:kolakoski}]
The tangram condition is necessary for every shuffle square.  If
$n\ge18$, it is sufficient by Theorem~\ref{thm:main}, since $P_n$ is
smooth.  The only tangram prefixes below length $18$ are
\[
\begin{array}{ccl}
P_4&=&1221,\\
P_8&=&12211212,\\
P_{16}&=&1221121221221121.
\end{array}
\]
Direct exhaustive checks show that the first two are not shuffle
squares and the last is one.  This proves the exact characterization.
Lemma~\ref{lem:infinite-parity} supplies infinitely many tangram
prefixes, all but finitely many of length at least $18$, proving the
final assertion.
\end{proof}

Call an infinite word a \emph{uniform shuffle square} if its positions
can be colored $U$ and $D$ so that the two induced infinite
subsequences are identical and the lengths of the maximal
monochromatic intervals are uniformly bounded.

\begin{corollary}\label{cor:uniform}
The infinite Kolakoski sequence is a uniform shuffle square.  It admits
such a coloring in which every monochromatic interval has length at
most $48$.
\end{corollary}

\begin{proof}
Include the empty prefix as a tangram boundary.  By
Lemma~\ref{lem:infinite-parity}, there are infinitely many further
tangram boundaries.  The factors between consecutive boundaries are
parity atoms.  They are smooth by factoriality, so
Lemma~\ref{lem:atoms} shows that each has even length at most $8$.

Starting from the beginning of $\K$, greedily concatenate consecutive
atoms until their total length is at least $18$.  Immediately before
the final atom the accumulated length is even and at most $16$, so each
resulting block has length between $18$ and $24$.  Every block is a
smooth tangram and hence, by Theorem~\ref{thm:main}, a shuffle square.

For every block choose a shuffle-square coloring and let $r_i$ be its
finite root.  Concatenating these colorings makes both induced infinite
subsequences equal to
\[
 r_1r_2r_3\cdots.
\]
Each block contains both colors.  A maximal monochromatic interval can
therefore meet at most two consecutive blocks and has length at most
$24+24=48$.  This proves uniformity.  The choices of the finite roots,
and hence the resulting infinite root, need not be unique.
\end{proof}

\section{Concluding remarks}

The proof isolates a general mechanism.  Suppose a factorial language
has uniformly bounded parity atoms, and shuffle-squareness is verified
on the finite interval needed to split both sides above a threshold.
Closure under concatenation then gives an eventual theorem.  For binary
smooth words the atom bound is $8$, the threshold is $18$, and the
entire obstruction set has only $34$ words.

It would be natural to ask whether analogous finite-obstruction results
hold for generalized smooth languages, for larger alphabets, or for
shuffle powers of order greater than two.

\section*{Acknowledgments}

The author thanks Jaros\l{}aw Grytczuk for correspondence, additional
references, and the observation that the argument yields a uniform
infinite shuffle square, formalized in Corollary~\ref{cor:uniform}.
The author also thanks Micha\l{} Zwierzy\'nski for sharing his independent
work and for discussions of the relationship between the two approaches.


\begin{thebibliography}{99}

\bibitem{BBC}
V.~Berth\'e, S.~Brlek, and P.~Choquette,
\newblock Smooth words over arbitrary alphabets,
\newblock \emph{Theoret. Comput. Sci.} 341 (2005), 293--310,
\newblock \url{https://doi.org/10.1016/j.tcs.2005.04.010}.


\bibitem{Bulteau}
L.~Bulteau and S.~Vialette,
\newblock Recognizing binary shuffle squares is NP-hard,
\newblock \emph{Theoret. Comput. Sci.} 806 (2020), 116--132.

\bibitem{Carpi}
A.~Carpi,
\newblock On repeated factors in $C^\infty$-words,
\newblock \emph{Inform. Process. Lett.} 52 (1994), 289--294.

\bibitem{CassaigneHenry}
J.~Cassaigne and R.~Henry,
\newblock The complexity of smooth words over binary alphabets,
\newblock arXiv:2603.10733, 2026.

\bibitem{DatkoPawlik}
D.~Datko and B.~Pawlik,
\newblock Roots of binary shuffle squares,
\newblock \emph{Symmetry} 17 (2025), Article 305.

\bibitem{DellaCorte}
A.~Della Corte,
\newblock Kolakoski sequence: links between recurrence, symmetry and
limit density,
\newblock \emph{Open J. Discret. Appl. Math.} 4 (2021), 29--44.

\bibitem{FedouFici}
J.-M.~Fédou and G.~Fici,
\newblock Automata and differentiable words,
\newblock \emph{Theoret. Comput. Sci.} 443 (2012), 46--62.

\bibitem{GPR}
J.~Grytczuk, B.~Pawlik, and A.~Ruci\'{n}ski,
\newblock Shuffle squares and ordered nest-free graphs,
\newblock arXiv:2503.22043v2, 2025.

\bibitem{HHNT}
X.~He, E.~Huang, I.~Nam, and R.~Thaper,
\newblock Shuffle squares and reverse shuffle squares,
\newblock \emph{European J. Combin.} 116 (2024), Article 103883,
\newblock \url{https://doi.org/10.1016/j.ejc.2023.103883}.

\bibitem{HePost}
X.~He and L.~Post,
\newblock Asymptotically half of binary words are shuffle squares,
\newblock arXiv:2512.12077, 2025,
\newblock \url{https://arxiv.org/abs/2512.12077}.

\bibitem{HRS}
D.~Henshall, N.~Rampersad, and J.~Shallit,
\newblock Shuffling and unshuffling,
\newblock \emph{Bull. Eur. Assoc. Theor. Comput. Sci.} 107 (2012),
131--142.

\bibitem{Rao}
M.~Rao,
\newblock Bifix codes and smooth words,
\newblock slides, Workshop on Challenges in Combinatorics on Words,
Fields Institute, April 2013,
\newblock \url{https://cs.uwaterloo.ca/~shallit/Fields/rao2.pdf}.

\bibitem{Weakley}
W.~D.~Weakley,
\newblock On the number of $C^\infty$-words of each length,
\newblock \emph{J. Combin. Theory Ser. A} 51 (1989), 55--62.

\bibitem{Zwierzynski}
M.~Zwierzy\'nski,
\newblock Shuffle squares in differentiable words,
\newblock unpublished manuscript, version of 27 September 2026,
privately communicated.

\end{thebibliography}
\end{document}